\documentclass[11pt,oneside,english]{amsart}
  \usepackage{amsmath}
  \usepackage{amsfonts}
  \usepackage{amsthm}
  \usepackage{amssymb}
  \usepackage[all]{xy}
  \usepackage{verbatim}  
  \usepackage{longtable}
  \usepackage{stmaryrd}   
  \usepackage{mathrsfs}   
  \usepackage{color}
  \usepackage{enumerate}
  \usepackage[center,loose,nooneline]{subfigure}
  
\makeatletter \theoremstyle{plain}
\newtheorem{thm}{Theorem}[section]

\newtheorem{prop}[thm]{Proposition}

\theoremstyle{remark}

\numberwithin{equation}{section}

\theoremstyle{definition}
\newtheorem{defin}[thm]{Definition}

 \newcommand{\N}{\mathbb{N}} 
 
 \newcommand{\R}{\mathbb{R}}

 \newcommand{\Hyp}{\mathbb{H}}
 \newcommand{\sphere}{\mathbb{S}}

\newcommand{\tand}{\text{ and }}

 \newcommand{\Hhh}{\mathscr{H}} 
 \newcommand{\Lll}{\mathscr{L}}

 \newcommand{\dm}{\mathrm{d}}
 
  \newcommand{\tPi}{\tilde{\Pi}}

 \newcommand{\Diag}{\mathrm{Diag}}

\begin{document}
\title[Projections of Riemannian surfaces]{Dimensions of  projections of sets on Riemannian surfaces of constant curvature}
\author{Zolt\'an M. Balogh, Annina Iseli}

\address {Mathematisches Institut,
Universit\"at Bern,
Sidlerstrasse 5,
CH-3012 Bern,
Switzerland}

\email{zoltan.balogh@math.unibe.ch}
\email{annina.iseli@math.unibe.ch}


\thanks{ This research was partially supported by the Swiss National Science Foundation.}

\begin{abstract}
We apply the theory of Peres and Schlag to obtain generic lower bounds for Hausdorff  dimension of images of sets by orthogonal projections on simply connected two-dimensional Riemannian manifolds of constant curvature. As a conclusion we obtain appropriate versions of Marstrand's theorem, Kaufman's theorem and Falconer's theorem in the above geometrical settings. 
\end{abstract}
\maketitle
\setcounter{equation}{0}

\section{Introduction} 
Since orthogonal projections are Lipschitz maps, they decrease the Hausdorff dimension of sets. For example, if we take a set $A \subset \R^2$ with $\dim A \leq 1$  then $\dim \Pi_{\theta} (A) \leq \dim A$ for all angles $\theta \in [0,\pi )$ where $\Pi_{\theta}: \R^2 \to L_{\theta}$ is the orthogonal projection onto the line through the origin in $\R^2$ which makes an $\theta$ with the $x$-axis. Marstrand \cite{Marstrand1954} and later Kaufman \cite{Kaufman1968} proved that that there is a {\it generic lower bound} on the dimension distortion, namely that the equality $\dim \Pi_{\theta} (A) = \dim A$ holds for almost every $\theta \in [0,\pi)$. An improvement of these result estimating the size of exceptional sets is due to Falconer \cite{Falconer1982}. For higher dimensional generalization and a unified exposition of this type of results we refer to the books \cite{Mattila1995}, \cite{Mattila2015},   as well as to the expository articles \cite{FFJ2015} and \cite{Mattila2004}.

\medskip

It is a purpose of general interest to extend the above results to various settings of non-Euclidean geometries. In this sense we mention the recent works \cite{BDCFMT2013, BFMT2012, Hovila2014} for the treatment of these questions in the setting of the Heisenberg groups. Due to the complicated sub-Riemannian geometry of the Heisenberg group the above mentioned results are much weaker and much less complete than their Euclidean counterparts. It is  expected that better 
results could be obtained in the setting of Riemannian manifolds. Various questions of geometric measure theory have been already been addressed in the setting of Riemannian manifolds. This includes the work of Brothers \cite{Brothers1966, Brothers1969}Ê in connection to Besicovitch-Federer type  characterization of purely unrectifiable sets in terms of projections  in the setting of homogenous spaces and also the more recent work of Hovila, J\"arvenp\"a\"a,  J\"arvenp\"a\"a and Ledrappiar \cite{HJJL2012, HJJL2} on two-dimensional Riemann surfaces. To our knowledge no Marstrand type result is yet available in the setting of curved geometries. The purpose of this note is a first step in this direction. 

Our main result shows that on simply connected two-dimensional Riemannian manifolds of constant curvature,  the same projection theorems hold as in the planar case. To formulate our main result we consider $M_K$ to be a two-dimensional simply connected Riemannian manifold with constant curvature $K$ and $p\in M_K$ be a fixed point. If $K\leq 0$ then the orthogonal projections $\Pi_{\theta}$ onto  geodesic lines $L_{\theta}$ emanating from $p$ are well defined in the whole space $M_K$. Here $L_{\theta}$ is the geodesic line in direction $\theta$ i.e. the image of the line $l_{\theta}\subset \R^2$ under the exponential map at $p$. If $K>0$ then the orthogonal projection $\Pi_{\theta}$ as above is only defined on compact sets $\Omega \subseteq B(p, \frac{\pi}{2\sqrt{K}})$. The main result of this note is formulated as follows:

\pagebreak

\begin{thm}\label{mainthm} Let $M_K$ be a complete, simply connected two-dimensional Riemannian manifold with constant curvature $K$, $p\in M_K$ a base point, and $\Omega$ be a compact subset of $M_K$. If $K>0$ we assume 
that  $\Omega \subseteq B(p, \frac{\pi}{2\sqrt{K}})$.  Denote by $\Pi_{\theta}$ the orthogonal projection onto the geodesic line $L_{\theta}$ emanating from $p$ in direction $\theta$.
Then for all Borel sets $A\subseteq \Omega$ the following statements hold. \begin{enumerate}
\item If $\dim A > 1$, then
\begin{enumerate}
\item  $\Lll^1 (\Pi_\theta A)>0$ for $\mathscr{L}^1$-a.e. $\theta\in (0,\pi)$.
\item $ \dim\{\theta \in (0,\pi)\ : \ \mathscr{L}^1(\Pi_\theta A)=0\}\leq 2-\dim A$.
\end{enumerate} 
\item If $\dim A \leq 1$, then
\begin{enumerate}
\item  $\dim (\Pi_\theta A)=\dim A$ for $\mathscr{L}^1$-a.e. $\theta\in (0,\pi)$.
\item For $0<\alpha\leq\dim A$, $\dim\{\theta \in (0,\pi)\ : \ \dim(\Pi_\theta A)<\alpha\}\leq \alpha.$
\end{enumerate}
\end{enumerate}
\end{thm}

Our proof is based on the theory of Peres and Schlag \cite{PS2000} which provides a
general abstract framework of generic Hausdorff dimension distortion results in metric spaces. The statements of Threorem \ref{mainthm} will follow by the verification of the crucial conditions of regularity and transversality of projections allowing the application of the 
results from \cite{PS2000}.  This is based on considerations using hyperbolic trigonometry for the case of negative curvature and spherical trigonometry in the case of positive curvature. 

The structure of the paper is as follows:
In the first section we recall the notation and the statement of the main result from \cite{PS2000} and reduce the statement of Theorem \ref{mainthm} to the hyperbolic and spherical case.  
In the second section we prove the statement of the main theorem in the hyperbolic case and in the third section we consider the spherical case. The last section is for final remarks. 

\medskip

{\bf Acknowledegments:} We thank the referee for carefully reading the paper and for helpful
remarks improving our presentation.

\section{Preliminaries}

We will now give a short summary of Peres and Schlag's theory \cite{PS2000} and recall one of their main results that we will apply to the Riemannian setting in the following sections. A nice summary of Peres and Schlag's work (inlcuding outlines of the main proofs) can also be found in \cite{Mattila2004} or \cite{Mattila2015}.\medskip

Let $(\Omega,\dm)$ be a compact metric space, $J\subset \R$ an open interval and $\Pi$ a continuous map
\begin{equation}\label{eq_Pi}
\Pi: J\times \Omega \rightarrow \R, \ \
(\lambda, \omega) \mapsto \Pi(\lambda,\omega).
\end{equation}
We think of $\Pi$ as a family of projections $\Pi_\lambda \omega:=\Pi(\lambda,\omega)$ over the parameter interval $J$. Let $\lambda\in J$ and $\omega_1,\omega_2\in \Omega$ two distinct points. We define
\begin{equation}\label{def_Phi}
\Phi_\lambda(\omega_1, \omega_2)=\frac{\Pi_\lambda \omega_1 - \Pi_\lambda \omega_2}{\dm(\omega_1,\omega_2)}.
\end{equation}

\begin{defin}\label{def_PS}\begin{enumerate}[(a)]
\item We say that $\Pi_\lambda$ has \emph{bounded derivatives in }$\lambda$, if: For all $\omega\in\Omega$ the function $\lambda \mapsto \Pi(\lambda,\omega)$ is smooth and for all compact intervals $I\subset J$ and all $l\in \N_0$, there exists a constant $C_{l,I}$ such that for all $\lambda \in I$ and $\omega \in \Omega$,
\begin{equation*}
\left| \frac{\dm^l}{\dm\lambda^l} \Pi(\lambda,\omega) \right| \leq C_{l,I}.
\end{equation*}
\item We call $J$ an \emph{interval of transversality of order $0$ for $\Pi$}, or shorter, the \emph{transversality property} is satisfied, if there exists a constant $C'>0$, such that for all pairs of distinct points $\omega_1,\omega_2\in \Omega$ and $\lambda\in J$,
\begin{equation*}
|\Phi_\lambda(\omega_1,\omega_2)|\leq C' \Rightarrow \left|\frac{\dm}{\dm \lambda}\Phi_\lambda(\omega_1,\omega_2)\right|\geq C'.
\end{equation*}
\item We say that $\Phi$ is \emph{$\infty$-regular}, if for each $l\in\N$ there exist a constant $C_{l}$ such that for all $\lambda \in J$ and distinct points $\omega_1,\omega_2\in \Omega$,
\begin{equation*}
\left| \frac{\dm^l}{\dm\lambda^l} \Phi_\lambda(\omega_1,\omega_2) \right| \leq C_{l}.
\end{equation*}
\end{enumerate}
\end{defin}
This definition allows us to state the following theorem due to Peres and Schlag \cite{PS2000}.

\begin{thm}\label{mainthm1}
Let $\Omega$ be a compact metric space which is bi-Lipschitz equivalent to a subset of  a Euclidean space; $J$ an open interval and $\Pi$ a continuous map as described in (\ref{eq_Pi}). Assume that conditions (a), (b) and (c) of Definition~\ref{def_PS} are satisfied.
Then the following statements hold for all Borel sets $A\subseteq \Omega$. \begin{enumerate}
\item If $\dim A > 1$, then 
\begin{enumerate}
\item  $\Lll^1 (\Pi_\lambda A)>0$ for $\mathscr{L}^1$-a.e. $\lambda\in  J$,
\item $ \dim\{\lambda \in  J\ : \ \mathscr{L}^1(\Pi_\lambda A)=0\}\leq 2-\dim A$.
\end{enumerate} 
\item If $\dim A \leq 1$, then 
\begin{enumerate}
\item  $\dim (\Pi_\lambda A)=\dim A$ for $\mathscr{L}^1$-a.e. $\lambda\in  J$,
\item For $0<\alpha\leq\dim A$, $\dim\{\lambda \in  J\ : \ \dim(\Pi_\lambda A)<\alpha\}\leq \alpha$.
\end{enumerate}
\end{enumerate}
\end{thm}

Theorem \ref{mainthm} will follow from Theorem \ref{mainthm1} once we show that for orthogonal projection on $M_{K}$ the conditions from Definition~\ref{def_PS} are satisfied. On the other hand, simply connected, complete two-dimensional Riemannian manifolds with constant curvature $K$ are isometric to $M_K^2=\Hyp^2$ endowed with the  metric $\dm_K=\frac{1}{\sqrt{-K}}\dm$, where $\dm$ denotes the hyperbolic metric on $\Hyp^2$ for $K< 0$ and $M_K^2=\sphere^2$ endowed with the  metric $\dm_K=\frac{1}{\sqrt{K}}\dm$, where $\dm$ denotes the usual spherical metric on  the $\sphere^2$ for $K>0$. This implies that it is enough to verify the conditions of Definition~\ref{def_PS}  for the cases of $\Hyp^2$ and $\sphere^2$.

\section{Projections in $\Hyp^2$}

\subsection{Geodesic projections in $\Hyp^2$}\label{secsetting_hyp}

Let $\mathbb{H}^2$ denote the hyperbolic plane and $\dm$ the hyperbolic metric on $\mathbb{H}^2$. Let $p$ be a fixed base point in $\mathbb{H}^2$ and $v_0$ a vector of length $1$ in the tangent plane $T_p\mathbb{H}^2$ of $\Hyp^2$ at $p$. We denote by $L_0^+$ the geodesic starting at $p$ in direction $v_0$ and by $L_0^-$ the geodesic starting at $p$ in the direction $-v_0$. This defines the geodesic line $L_0=L_0^+\cup L_0^-$ through $p$.
For all angles $\theta \in (0,\pi)$ define $v_\theta$ to be the unique vector of length $1$ in $T_p\mathbb{H}^2$ such that the counter-clockwise angle from $v_0$ to $v_\theta$ is $\theta$. 
Let $L_\theta^+$ be the geodesic starting from $p$ in direction $v_\theta$ and $L_\theta^-$ be the geodesic starting from $p$ in direction $-v_\theta$. This defines the geodesic line $L_\theta=L_\theta^+\cup L_\theta^-$.\medskip

For a point $q\in \mathbb{H}^2$, let $P_\theta q$ be the unique point on $L_\theta$ that minimizes the distance between $L_\theta$ and $q$. In other words, $P_\theta q$ is the unique point of $L_\theta$ that satisfies,
\begin{equation*}
\dm(q,P_\theta q) = \inf\{\dm(q,q'):q'\in L_\theta \}.
\end{equation*}
The existence and uniqueness of such a point $P_\theta q$ holds in general for negatively curved spaces (see e.g. Proposition 2.4 in \cite{BH1999}, page 176). This allows us to define the mapping $P_\theta$:
\begin{equation*} 
P_\theta: \mathbb{H}^2  \rightarrow L_\theta,\ \
q  \mapsto P_\theta q.
\end{equation*}
Proposition 2.4 of \cite{BH1999} implies, that $P_\theta$ is distance non-increasing and that for each $q\in\mathbb{H}^2$  the geodesic connecting $q$ to $P_\theta q$ is orthogonal to $L_\theta$. Therefore, we will refer to the mapping $P_\theta$ as the \emph{orthogonal projection} of $\mathbb{H}^2$ onto $L_\theta$. \medskip

In order to be consistent with the notion of projection used in \cite{PS2000} we define the generalized projection
\begin{equation}\label{def_Pi_in_H}
\Pi: (0,\pi)\times \mathbb{H}^2 \rightarrow \R, \ \ 
(\theta, q)  \mapsto \Pi_\theta q:=\pm\dm(p,P_\theta q)
\end{equation}
where the sign "$\pm$" is to be understood as follows:\[ \Pi_\theta q=\dm(p,P_\theta q)\text{ if } P_\theta q\in L_\theta^+ \text{, and }\Pi_\theta q=-\dm(p,P_\theta q)\text{ if }P_\theta q\in L_\theta^-.\]
Note that it is immediate from the definition of $\Pi_\theta$ and $P_\theta$ that \begin{equation}\label{compare_P_Pi_hyp}
\dm(P_\theta p_1, P_\theta p_2)=\dm_{Eucl.}(\Pi_\theta p_1,\Pi_\theta p_2),
\end{equation}
for all $\theta\in(0,\pi)$ and $p_1,p_2\in \mathbb{H}^2$, where $\dm_{Eucl.}$ denotes the Euclidean metric on $\R$. Moreover, note that $\Pi$ is a continuous map as described in (\ref{eq_Pi}). The interval $J$ of parameters $\lambda$ from (\ref{eq_Pi}), here is an interval $(0,\pi)$ of angles $\theta$. The fact that $P_\theta$, for all $\theta\in(0,\pi)$, is a distance non-increasing mapping, implies that $\Pi_\theta$ is distance non-increasing, i.e., $1$-Lipschitz, for all $\theta\in(0,\pi)$. In particular, this implies that the dimension of a set can not increase under the projection $\Pi_q $, $q\in\Hyp^2$.\medskip

In order to express $\Pi_\theta$ in a way that allows us to study its transversality and regularity properties, we use basic facts from hyperbolic trigonometry. 
Consider a geodesic triangle in $\Hyp^2$ with side lengths $a,b,c$ and opposite angles $\alpha, \beta, \gamma$. It holds that
\begin{equation}\label{LoC1}
\cosh a= \cosh b \cosh c -\sinh b \sinh c \cos \alpha.
\end{equation} 
This formula is called the \emph{hyperbolic law of cosines}, a proof can be found for example in \cite{BH1999} or \cite{Thurston1997}.
Applying the hyperbolic law of cosines to a right-sided triangle twice, yields
\begin{equation}\label{LoC2}
 \tanh b =\tanh c\cos \alpha,
\end{equation} 
where $\gamma=\frac{\pi}{2}$. To see this, consider a triangle as just described with $\gamma=\frac{\pi}{2}$. From (\ref{LoC1}) it follows that $\cosh c = \cosh b \cosh a$ and $\cosh a = \cosh b \cosh c -\sinh b \sinh c \cos \alpha.$ From these relations we obtain
$\frac{\cosh c}{\cosh b}  = \cosh b \cosh c -\sinh b \sinh c \cos \alpha$, which implies
$-\frac{\cosh c}{\cosh b}\sinh^2 b =-\sinh b \sinh c \cos \alpha$. Thus, (\ref{LoC2}) follows. 

Now for each point $q\in \mathbb{H}^2$ and angle $\theta \in [0,\pi)$, let us denote by $\alpha_{q,\theta}\in [0,2\pi)$ the counter-clockwise angle from $L_\theta^+$ to the geodesic segment connecting the base point $p$ to $q$. As we will show now, (\ref{LoC2}) implies that
\begin{equation}\label{pre-calcproj}
\tanh \Pi_\theta q=\tanh\dm(p,q)\cos (\alpha_{q,\theta}),
\end{equation}
for all angles $\theta\in(0,\pi)$ and all points $q\in\mathbb{H}^2$. Let $q$ be a point in $\in\mathbb{H}^2$ and $\theta \in [0,\pi)$ an angle.
First, we consider the case when $0\leq\alpha_{q,\theta}<\frac{\pi}{2}$. Then, $P_\theta(q)\in L_\theta^+$ and the three points $p$, $q$ and $P_\theta q$ span a geodesic triangle with side lengths $a=\dm(q,P_\theta q)$, $ b=\dm(p,P_\theta q)$, $c=\dm(p,q)$ and opposite angles $\alpha=\alpha_{q,\theta}$,  $\beta$, $\gamma=\frac{\pi}{2}$.
By (\ref{LoC2}), it follows that $\tanh \dm(p,P_\theta q)=\tanh\dm(p,q)\cos (\alpha_{q,\theta})$.
Hence, by the definition of $\Pi_\theta$ and the fact that $P_\theta(q)\in L_\theta^+$, we obtain (\ref{pre-calcproj}) for this case. The other cases: $\frac{\pi}{2}\leq\alpha_{q,\theta}<\pi$;  $\pi\leq\alpha_{q,\theta}<\frac{3\pi}{2}$  and  $\frac{3\pi}{2}\leq\alpha_{q,\theta} <2\pi $ can be treated similarly.\medskip

For each point $q\in\Hyp^2$, let $\theta_q \in [0,2\pi)$ be the counter-clockwise angle from $L_0^+$ to the geodesic segment connecting the base point $p$ to $q$. It is easy to see that $\cos (\alpha_{q,\theta})= \cos (\theta_q-\theta)$ for all $\theta\in (0,\pi)$. In conclusion: 
\begin{equation}\label{calcproj}
\tanh \dm(p,P_\theta q)=\tanh\dm(p,q)\cos (\theta_q-\theta).
\end{equation}

Motivated by this result, we introduce the following new family of generalized projections: \begin{equation}\label{def_Pi_in_H_new}
\tilde{\Pi}: (0,\pi)\times \mathbb{H}^2 \rightarrow \R, \ \
(\theta, q)  \mapsto \tPi_\theta q := \tanh \dm(p, q)\cos(\theta_q-\theta).
\end{equation}
Note that, for all $\theta\in(0,\pi)$ and $q\in\mathbb{H}^2$, \begin{equation}\label{compare_Pi_tPi_H}
\tPi_\theta q=\tanh(\Pi_\theta q).
\end{equation}
Thus,  $\tPi:(0,\pi)\times \Omega\rightarrow \R$ is a continuous mapping with respect to $\dm$. Moreover, note that $\tanh$ is $1$-Lipschitz on the whole of $\R$. Recall, that for all $\theta\in(0,\pi)$, $\Pi_\theta$ is $1$-Lipschitz. Therefore, $\tPi_\theta$ is $1$-Lipschitz for all $\theta \in (0,\pi)$.\medskip

Now for all angles $\theta\in (0,\pi)$ and all pairs of distinct points $p_1,p_2\in \mathbb{H}^2$ define,
\begin{equation}\label{def_Phi_hyp}
\Phi_\theta(p_1,p_2)=\frac{\tPi_\theta p_1 - \tPi_\theta p_2}{\dm(p_1,p_2)},
\end{equation}
analogous to (\ref{def_Phi}) in the general setting.

\subsection{Transversality and regularity properties in $\Hyp^2$}\label{sec_trans_hyp}
Let $\Omega$ be a compact subset of $\mathbb{H}^2$. From now on we will consider the metric space $(\Omega, \dm)$, where $\dm$ denotes the restriction of the hyperbolic metric to $\Omega$. We will consider the projections $\Pi$ and $\tPi$ as defined in (\ref{def_Pi_in_H}) and (\ref{def_Pi_in_H_new}), as well as the function $\Phi$ as defined in (\ref{def_Phi_hyp}), restricted to $\Omega$.\par

We will now show that Definition~\ref{def_PS} is satisfied in this just defined setting. For this purpose, define $\Diag:=\{(p_1,p_2)\in\Omega\times\Omega \ : \ p_1=p_2\}$.

\begin{prop}\label{prop_claim_1_2}
There exist two functions
\begin{eqnarray*}
D \ : \ (\Omega\times\Omega)\backslash \Diag &\rightarrow & \R_+\\
\hat{\theta} \ : \ (\Omega\times\Omega)\backslash \Diag &\rightarrow & [0,2\pi),
\end{eqnarray*} such that:
\begin{enumerate}
\item For all pairs of points $(p_1,p_2)\in(\Omega\times\Omega)\backslash \Diag$ and all angles $\theta \in (0,\pi)$, \begin{equation*}
\tPi_\theta p_1 -\tPi_\theta p_2 = D(p_1,p_2)\cos(\theta-\hat{\theta}(p_1,p_2)).
\end{equation*}
\item There exist constants $c>0$ and $C>0$, such that for all $(p_1,p_2)\in(\Omega\times\Omega)\backslash \Diag$,
\begin{equation*}
c\leq  \frac{D(p_1,p_2)}{\dm(p_1,p_2)}\leq C.
\end{equation*}
\end{enumerate}
\end{prop}

\begin{proof}[Proof of Proposition~\ref{prop_claim_1_2}] 
Let $(p_1,p_2)\in(\Omega \times \Omega)\backslash\Diag$. Throughout this proof, we will use the following notation:
\begin{equation}\label{not_d_tilde}
d_1=\dm(p,p_1),\ d_2=\dm(p,p_2),\ d=\dm(p_1,p_2),\ \tilde{d_1}=\tanh \dm(p_1,p),\ \tilde{d_2}=\tanh \dm(p_2,p).
\end{equation}
Moreover, we denote the counter-clockwise angle from $L_0^+$ to the geodesic segment connecting $p$ to $p_1$ (resp. $p_2$) by $\theta_1$ (resp. $\theta_2$).\par 
By (\ref{pre-calcproj}), we have $ \tPi_\theta p_1 =  \tilde{d_1} \cos (\theta -\theta_1)$ and $\tPi_\theta p_2=  \tilde{d_2} \cos (\theta-\theta_2).$ In order to make the calculations clearer, write $\alpha=\theta-\theta_2$ and $\alpha_0=\theta_1-\theta_2$. Thus we obtain
\begin{equation}\label{eq2}
\tPi_\theta p_1= \tilde{d_1} \cos (\alpha-\alpha_0),\ \ 
\tPi_\theta p_2 =\tilde{d_2} \cos (\alpha).
\end{equation}
and by an elementary calculation
\begin{equation}\label{eq4}
\tPi_\theta p_1 -\tPi_\theta p_2 = (\tilde{d_1} \cos \alpha_0 -\tilde{d_2})\cos\alpha + \tilde{d_1} \sin\alpha_0\sin\alpha.
\end{equation}

Define \begin{equation}\label{AB}
A = \tilde{d_1} \cos \alpha_0 - \tilde{d_2}, \ \
B = \tilde{d_1} \sin\alpha_0.
\end{equation} 
Note that $A$ and $B$ cannot both be $0$, since $(p_1,p_2)\notin \Diag$.
This allows us to make the following definition: Let $\hat{\alpha}\in(0,2\pi)$ be the angle that satisfies
\begin{equation}\label{thetahat}
\cos \hat{\alpha}= \frac{A}{\sqrt{A^2+B^2}} \ \ \tand \  
\sin \hat{\alpha}= \frac{B}{\sqrt{A^2+B^2}}.
\end{equation}

In this notation, from (\ref{eq4}) it follows that $\tPi_\theta p_1 -\tPi_\theta p_2 =\sqrt{A^2+B^2} \cos(\alpha-\hat{\alpha})$. Set $\hat{\theta}=\theta_2 +\hat{\alpha}$ (see below (\ref{not_d_tilde}) for the definition of $\theta_2$)  and $D=\sqrt{A^2+B^2}$. Observe that by their definition both $D$ and $\hat{\theta}$ are independent of $\theta$.  Thus $D=D(p_1,p_2)$ and $\hat{\theta}=\hat{\theta}(p_1,p_2)$ are well-defined functions on $(\Omega\times \Omega)\backslash \Diag$. Moreover, by definition of $\alpha, \hat{\alpha}$ and $\hat{\theta}$, we conclude
\begin{equation*}
\tPi_\theta p_1 - \tPi_\theta p_2=D\cos(\theta-\hat{\theta}).
\end{equation*}

This completes the proof of Proposition~\ref{prop_claim_1_2}.$(1)$. 

For the proof of Proposition~\ref{prop_claim_1_2}.$(2)$ it suffices to show that $c\leq \frac{D(p_1,p_2)}{\dm(p_1,p_2)}\leq C$ for constants $c>0$ and $C>0$ independent of $p_1$ and $p_2$, where ${D(p_1,p_2)=\sqrt{A^2+B^2}}$.\medskip

By the hyperbolic law of cosines (\ref{LoC1}) applied to the geodesic triangle spanned by $p$,$p_1$ and $p_2$, it holds that $\cosh d=\cosh d_1\cosh d_2-\sinh d_1\sinh d_2\cos\alpha_0$, 
which implies,
\begin{equation}\label{eq8}
-2\tanh d_1 \tanh d_2\cos\alpha_0= 2\left( \frac{\cosh d}{\cosh d_1 \cosh d_2}-1 \right).
\end{equation}
Applying (\ref{AB}) and (\ref{eq8}), as well as elementary product-to-sum identities for hyperbolic and trigonometric functions, yields
\begin{equation}\label{eq_A^2+B^2}
A^2+B^2 =  \frac{2 \cosh d \cosh d_1 \cosh d_2 -\cosh^2 d_1 - \cosh^2 d_2}{\cosh^2 d_1\cosh^2 d_2}.
\end{equation}

Note that the product $\cosh d_1 \cosh d_2$ is greater than $1$ and is bounded from above since $p_1,p_2\in \Omega$ and $\Omega$ is compact. So we can derive the following upper bound for $A^2+B^2$:
\begin{equation*}
A^2+B^2 \leq  \left(\frac{1}{\cosh^2 d_1}+\frac{1}{\cosh^2 d_1}\right)(\cosh d -1) \leq 2(\cosh d-1).
\end{equation*}
Hence, we conclude that
\begin{equation*}
\frac{\sqrt{A^2+B^2}}{d}\leq \sqrt{2} \frac{\sqrt{\cosh d-1}}{d}.
\end{equation*}

Note that $\frac{\sqrt{\cosh d-1}}{d}$ is a continuous function in $d>0$ and that $\lim_{d\to 0^{+}} \frac{\sqrt{\cosh d-1}}{d}=\frac{1}{\sqrt{2}}<\infty.$
Thus by the compactness of $\Omega$, we have $\frac{\sqrt{A^2+B^2}}{d}\leq C$ for some constant $C>0$ only depending on the diameter of $\Omega$. This proves the right-hand inequality in Proposition~\ref{prop_claim_1_2}.$(2)$. Now let us prove the left-hand inequality.\par 
Using the notation from (\ref{not_d_tilde}), we define $\rho=d_1-d_2$. By the triangle inequality $\rho\in [-d,d]$, i.e., $|d|\geq |\rho|$ and therefore $\cosh d \geq \cosh \rho$. The following calculation only uses the definition of $\rho$ and elementary calculation rules for $\cosh$:
\begin{equation*}\begin{split}
&2 \cosh d \cosh d_1 \cosh d_2 -\cosh^2 d_1 - \cosh^2 d_2\\
&= \ 2 \cosh d \cosh (d_2+\rho) \cosh d_2 -\cosh^2 (d_2+\rho) - \cosh^2 d_2\\
&= \ \cosh d (\cosh(2 d_2 +\rho)+\cosh \rho)- \frac{1}{2}(\cosh(2(d_2+\rho))+1) - \frac{1}{2}(\cosh(2d_2)+1) \\
&= \ \cosh d (\cosh(2 d_2 +\rho)+\cosh \rho) - \frac{1} {2}(\cosh(2(d_2+\rho)+\cosh(2d_2)) -1\\
&= \ \cosh d (\cosh(2 d_2 +\rho)+\cosh \rho)- \cosh(2 d_2+\rho)\cosh \rho -1\\
&= \cosh d \cosh \rho-1+(\cosh d-\cosh \rho)  \cosh(2 d_2+\rho)\\
&\geq \cosh d \cosh \rho-1\geq \cosh d -1.\end{split}
\end{equation*}

From the Taylor series representation of $\cosh$ it follows that $ \cosh d -1 \geq \frac{1}{2}d^2$. Consequently, the estimate, \begin{equation}\label{eq11} 2 \cosh d \cosh d_1 \cosh d_2 -\cosh^2 d_1 - \cosh^2 d_2\geq \frac{1}{2}d^2,
\end{equation} follows. Now, since $p_1,p_2\in \Omega$ and $\Omega$ compact, there exists a constant $\tilde{c}>0$ (only depending on $\Omega$) such that $\frac{1}{\cosh^2 d_1 \cosh^2 d_2}\geq \tilde{c}$.
Thus by (\ref{eq_A^2+B^2}) and (\ref{eq11}), it follows that
$\frac{\sqrt{A^2+B^2}}{d}\geq c$ for $c=\sqrt{\frac{\tilde{c}}{2}}$. This concludes the proof of Proposition~\ref{prop_claim_1_2}. \end{proof}

\begin{proof}[Proof of Theorem \ref{mainthm} in the negative curvature case:] 
From Proposition~\ref{prop_claim_1_2}.$(1)$, it follows that for all pairs of points $(p_1,p_2)\in(\Omega\times\Omega)\backslash \Diag$ and angle $\theta \in (0,\pi)$,
 $\Phi_\theta(p_1,p_2)=\frac{D(p_1,p_2)}{\dm(p_1,p_2)}\cos(\theta-\hat{\theta}(p_1,p_2))$ and hence
 \begin{equation}\label{sin_deriv}
\frac{\dm}{\dm \theta}\left( \tPi_\theta p_1 -\tPi_\theta p_2 \right)=- D(p_1,p_2)\sin(\theta-\hat{\theta}(p_1,p_2)).
\end{equation} 

Thus for all $l\in \N$, $\frac{\dm^l}{\dm \theta^l} \Phi_\theta(p_1,p_2)$ is an element of the set
\begin{equation}\label{deriv}
\left\{\pm \frac{D(p_1,p_2)\sin(\theta-\hat{\theta}(p_1,p_2))}{\dm(p_1,p_2)}, \pm  \frac{D(p_1,p_2)\cos(\theta-\hat{\theta}(p_1,p_2))}{\dm(p_1,p_2)} \right\}
\end{equation} 
Consequently, from Proposition~\ref{prop_claim_1_2}.$(2)$ it follows that $\Phi_\theta$ is $\infty$-regular and has bounded partial derivatives in the sense of Definition~\ref{def_PS}. Now let $c'>0$ such that $c'<\frac{c}{10}$, where $c$ is the constant from Proposition~\ref{prop_claim_1_2}.$(2)$. Assume that $|\Phi_\theta(p_1,p_2)|\leq c'$. Applying Proposition~\ref{prop_claim_1_2}, yields
\[  |\cos(\theta-\hat{\theta}(p_1,p_2)) |\leq c'\frac{\dm(p_1,p_2)}{D(p_1,p_2)}\leq \frac{c'}{c}<\frac{1}{10},\] 
and hence, $|\sin(\theta-\hat{\theta}(p_1,p_2)) |\geq \frac{1}{10}$. Now by (\ref{sin_deriv}), it follows that $\left|\frac{\dm}{\dm \theta}\Phi_\theta(p_1,p_2)\right|\geq \frac{c}{10}$. Thus the transversality property holds as well. Now, by applying Theorem~\ref{mainthm1}, Theorem~\ref{mainthm} follows for the case when $\Omega$ is a compact subset of $\Hyp^2$. As explained in 
Section 2, the statement of Theorem \ref{mainthm} in the negative curvature case follows from this. 
\end{proof}

\section{Projections in $\sphere^2$}

\subsection{Geodesic projections in $\sphere^2$}\label{secsetting_sph}

Let $\sphere^2$ denote the Euclidean two-sphere equipped with the usual spherical metric $\dm$. Let $p$ be a fixed base point in $\sphere^2$, $m\in (0,\frac{\pi}{2})$ a fixed number and denote by $B(p,m)$ the open ball of radius $m$ centered at $p$.  Let $v_0$ be a vector of length $1$ in the tangent plane $T_p\sphere^2$ of $\sphere^2$ at $p$. We denote by $L_0^+$ the segment of the geodesic starting at $p$ in direction $v_0$ that is contained in $B(p,m)$. Analogously, denote by $L_0^-$ the segment of the geodesic starting at $p$ in the direction $-v_0$ that is contained in $B(p,m)$. This defines the geodesic segment $L_0=L_0^+\cup L_0^- \subset B(p,m)$ through $p$.
For an angle $\theta \in (0,\pi)$ let $v_\theta$ be the unique vector of length~$1$ in $T_p\sphere^2$ such that the counter-clockwise angle from $v_0$ to $v_\theta$ is $\theta$. 
Let $L_\theta^+$ be the segment of the geodesic starting from $p$ in direction $v_\theta$ that is contained in $B(p,m)$. Analogously define $L_\theta^-$ in direction $-v_\theta$. This defines the geodesic segment $L_\theta=L_\theta^+\cup L_\theta^-\subset B(p,m)$. Note that for each direction $v\in T_p\sphere^2$ there exists a geodesic line starting at $p$ in direction $v$ of length $\pi$. So the restriction onto $B(p,m)$ with $m<\frac{\pi}{2}$ might look too strong at this point. However, this restriction is crucial in order for our results to hold. We will explain this in more detail in the last section.\medskip

Let $\Omega\subset \sphere^2$ be a compact set that is contained in $B(p,m)$. Then, due to the restriction $m<\frac{\pi}{2}$, the orthogonal projection $P_\theta$ of $\Omega$ onto the geodesic line segment $L_\theta$ is well-defined by, 
\begin{equation*}
\dm(q,P_\theta q) = \inf\{\dm(q,q'):q'\in L_\theta \}.
\end{equation*}
(See \cite{BH1999}, pages 176-178.) By the same argument as in the hyperbolic plane, for a point $q\in\Omega$, the geodesic segment connecting $q$ to $P_\theta q$ is orthogonal to $L_\theta$. On the other hand $P_\theta$ is not $1$-Lipschitz. However, $P_\theta: \Omega\rightarrow L_\theta$, for all $\theta\in(0,\pi)$, still is a Lipschitz map for some constant that only depends on $m$.
\medskip

Define the generalized projection $\Pi$, analogously to (\ref{def_Pi_in_H}):
\begin{equation}\label{def_Pi_in_sph}
\Pi: (0,\pi)\times \Omega \rightarrow \R,\ \
(\theta, q)  \mapsto \Pi_\theta q:=\pm\dm(p,P_\theta q).
\end{equation}
It is immediate from this definition that \begin{equation}\label{compare_P_Pi_sph}
\dm(P_\theta p_1, P_\theta p_2)=\dm_{Eucl.}(\Pi_\theta p_1,\Pi_\theta p_2).
\end{equation}

In our considerations below we will use basic results of spherical trigonometry. The following formula is what we call the \emph{spherical law of cosines}, a proof can be found for example in \cite{BH1999} or \cite{Thurston1997}.\\
For a geodesic triangle with side lengths $a,b,c$, each $<\pi$, and opposite angles $\alpha, \beta, \gamma$, it holds that:
\begin{equation}\label{LoC1_sph}
\cos a= \cos b \cos c +\sin b \sin c \cos \alpha.
\end{equation} 

Applying the spherical law of cosines to a right-sided triangle twice, yields
\begin{equation}\label{LoC2_sph}
 \tan b =\tan c\cos \alpha,
\end{equation} 
where $\gamma=\frac{\pi}{2}$. (Note that (\ref{LoC2_sph}) can be proved similarly to (\ref{LoC2}).) For each point $q\in\Omega$, define the angle $\theta_q$ as in the hyperbolic plane (see above (\ref{calcproj})). Applying an argument similar to the proof of (\ref{calcproj}), yields that
\begin{equation}\label{pre-calcproj_sph}
\tan \Pi_\theta q=\tan(\dm(p,q))\cos(\theta-\theta_q).
\end{equation}
Motivated by (\ref{pre-calcproj_sph}), we define a new family of generalized projections:
\begin{equation}\label{def_Pi_in_Sph_new}
\tilde{\Pi}: (0,\pi)\times \Omega \rightarrow \R, \ \
(\theta, q)  \mapsto \tPi_\theta q:=\tan(\dm(p,q))\cos(\theta-\theta_q). 
\end{equation} (Compare (\ref{pre-calcproj}) and (\ref{def_Pi_in_H_new}).) Note that for all $\theta\in(0,\pi)$ and $q\in\Omega$, \begin{equation}\label{compare_Pi_tPi_Sph}
\tPi_\theta=\tan(\Pi_\theta)\ .
\end{equation}
Thus, $\tPi$ is continuous with respect to $\dm$ and for all $\theta\in(0,\pi)$, $\tPi_\theta$ is Lipschitz, for some Lipschitz constant that only depends on $m$.\par

Now for all angles $\theta\in (0,\pi)$ and all pairs of distinct points $p_1,p_2\in \Omega$ define,
\begin{equation*}
\Phi_\theta(p_1,p_2)=\frac{\tPi_\theta p_1 - \tPi_\theta p_2}{\dm(p_1,p_2)}.
\end{equation*}

\subsection{Transversality and regularity properties in $\sphere^2$}\label{sec_trans_sph}

We will now show that Definition~\ref{def_PS} is satisfied in the setting described in Section~\ref{secsetting_sph}.\par

\begin{prop}\label{prop_claim_1_2_sph}
There exist two functions
\begin{eqnarray*}
D \ : \ (\Omega\times\Omega)\backslash \Diag &\rightarrow & \R_+\\
\hat{\theta} \ : \ (\Omega\times\Omega)\backslash \Diag &\rightarrow & [0,2\pi),
\end{eqnarray*}
such that:
\begin{enumerate}
\item For all pairs of points $(p_1,p_2)\in(\Omega\times\Omega)\backslash \Diag$ and angle $\theta \in (0,\pi)$, \begin{equation*}
\tPi_\theta p_1 -\tPi_\theta p_2 = D(p_1,p_2)\cos(\theta-\hat{\theta}(p_1,p_2)).\end{equation*}
\item Moreover, there exist constants $c>0$ and $C>0$, such that for all $(p_1,p_2)\in(\Omega\times\Omega)\backslash \Diag$ 
\begin{equation*}
c\leq  \frac{D(p_1,p_2)}{\dm(p_1,p_2)}\leq C.
\end{equation*}
\end{enumerate}
\end{prop}

\begin{proof}[Proof of Proposition~\ref{prop_claim_1_2_sph}.]
Let $(p_1,p_2)\in(\Omega \times \Omega)\backslash\Diag$. Throughout this proof, we will use the following notation: \begin{equation}\label{not_d_tilde_sph}
d_1=\dm(p,p_1),\ d_2=\dm(p,p_2),\ d=\dm(p_1,p_2),\ \tilde{d_1}=\tan \dm(p_1,p),\ \tilde{d_2}=\tan \dm(p_2,p).
\end{equation}
Moreover, we denote the counter-clockwise angle from $L_0^+$ to the geodesic segment connecting $p$ to $p_1$ (resp. $p_2$) by $\theta_1$ (resp. $\theta_2$). With this notation, the proof of Proposition~\ref{prop_claim_1_2_sph}.$(1)$ is similar to the proof of Proposition~\ref{prop_claim_1_2}.$(1)$.

In order to prove Proposition~\ref{prop_claim_1_2_sph}.$(2)$ it suffices to show that $c\leq \frac{\sqrt{A^2+B^2}}{d}\leq C$, for constants $c>0$ and $C>0$ independent of $p_1$ and $p_2$. Recall that $A$ and $B$ are defined as \begin{equation}\label{AB_s}
A = \tilde{d_1} \cos \alpha_0 - \tilde{d_2}\ \tand \ 
B = \tilde{d_1} \sin\alpha_0,
\end{equation} 
where $\alpha_0=\theta_1-\theta_2$, see (\ref{eq2}) and (\ref{AB}). \par 

By the spherical law of cosines (\ref{LoC1_sph}), it holds that
\begin{equation*}
\cos d=\cos d_1\cos d_2+\sin d_1\sin d_2\cos\alpha_0.
\end{equation*}  
Since $d_1$ and $d_2$ are both strictly smaller than $\frac{\pi}{2}$, $\cos d_1 \cos d_2 \neq 0$, and we obtain \begin{equation}\label{eq8_s}
-2\tan d_1 \tan d_2\cos\alpha_0= 2\left( 1-\frac{\cos d}{\cos d_1 \cos d_2}\right).
\end{equation}
From (\ref{AB_s}), (\ref{eq8_s}) and elementary calculation rules for trigonometric functions it follows that \begin{equation}\label{eq9_s}
A^2+B^2 =\frac{\cos^2 d_1 + \cos^2 d_2-2 \cos d \cos d_1 \cos d_2}{\cos^2 d_1  \cos^2 d_2}.
\end{equation}

Using the fact that $d_1,d_2\in (0,\frac{\pi}{2})$ and thus $0<\cos d_1, \cos d_2<1$, we can derive the following lower bound for $A^2+B^2$:
\begin{equation*}
A^2+B^2\geq \frac{2\cos d_1\cos d_2-2 \cos d \cos d_1 \cos d_2}{\cos^2 d_1  \cos^2 d_2}= \frac{2(1-\cos d)}{\cos d_1\cos d_2}\geq 2(1-\cos d).
\end{equation*}
This implies that
\begin{equation}\label{eq10_s}
\frac{\sqrt{A^2+B^2}}{d}\geq \sqrt{2}\frac{\sqrt{1-\cos d}}{d}.
\end{equation}

The function $d\mapsto\frac{\sqrt{1-\cos d}}{d}$ is continuous on $(0,\infty)$ and $
\lim_{d\to 0^{+}} \frac{\sqrt{1-\cos d}}{d}=\frac{1}{\sqrt{2}}>0.$ 
Since $0<d<2m<\pi$, it follows that there exists a constant $c$, only depending on $m$, such that $\sqrt{2}\frac{\sqrt{1-\cos d}}{d}\geq c$. This together with (\ref{eq10_s}) proves the left-hand inequality in Proposition~\ref{prop_claim_1_2_sph}.$(2)$.

Now let us prove the right-hand inequality. We define $\rho=d_1-d_2$, thus by the triangle inequality $0<|\rho|\leq |d|<\pi$ and therefore $
\cos d \leq \cos \rho$. The following calculation only uses the definition of $\rho$ and elementary calculation rules for $\cos$::
\begin{equation*}\begin{split}
&\cos^2 d_1 + \cos^2 d_2-2 \cos d \cos d_1 \cos d_2\\ 
&= \ \cos^2 (d_2+\rho) + \cos^2 d_2 - 2 \cos d \cos (d_2+\rho) \cos d_2\\
&= \ \frac{1}{2}(\cos (2(d_2+\rho))+1) + \frac{1}{2}(\cos (2d_2)+1) - \cos d (\cos (2 d_2 +\rho)+\cos  \rho)\\
&=\  1+ \frac{1} {2}(\cos (2(d_2+\rho))+\cos (2d_2)) - \cos d (\cos (2 d_2 +\rho)+\cos \rho)\\
&= \ 1+ \cos (2 d_2+\rho)\cos  \rho  -\cos d (\cos (2 d_2 +\rho)+\cos \rho)\\
&= \ 1- \cos d \cos \rho +(\cos \rho-\cos d ) \cos (2 d_2+\rho)\\ 
&\leq  \ 1-\cos d \cos \rho + (\cos\rho -\cos d)
\ \leq  \ 2(1- \cos d).\end{split}
\end{equation*}

Note that $\ 2(1- \cos d)
\ \leq  \ d^2$ for $0<d<2m<\pi$. Consequently, the estimate, 
\begin{equation}\label{eq11_s}
\cos^2 d_1 + \cos^2 d_2-2 \cos d \cos d_1 \cos d_2\leq d^2
\end{equation} 
follows. Recall that $d_1,d_2<m$. Set $C=\frac{1}{\cos^4 m}$, then $\frac{1}{\cos^2 d_1 \cos^2 d_2}\leq C$ and hence, by (\ref{eq9_s}) and (\ref{eq11_s}), we obtain $\frac{\sqrt{A^2+B^2}}{d}\leq C$.\end{proof}


\begin{proof}[Proof of Theorem \ref{mainthm} in the positive curvature case:] By applying Proposition \ref{prop_claim_1_2_sph} (analogously to the application of Proposition~\ref{prop_claim_1_2} in the proof of Theorem~\ref{mainthm} in the negative curvature case)  and Theorem~\ref{mainthm1}, Theorem~\ref{mainthm} follows for the case of the projections on the Euclidean sphere $\sphere^2$. As explained in 
Section 2, the statement of Theorem \ref{mainthm} in the positive curvature case follows. 
\end{proof}

\section{Final remarks} 

It is clear that the compactness of $\Omega \subset M_K$ is not an essential condition in 
Theorem \ref{mainthm} in the case when 
$K< 0$. Indeed, any set $A\subset M_K$ can be included in a countable union of compact subsets $\{\Omega_k\}_{k\in \N}$. 
Applying the statements of the theorem for $A \cap  \Omega_k$ they follow for $A$ as well. \medskip

By a similar argument it can be shown that also in the case of $K>0$ the compactness of $\Omega$ is not essential for Theorem~\ref{mainthm} to hold. However, the restriction $\Omega \subset B(p,\frac{\pi}{2K})$ is essential. 
To see this let us consider the case of the sphere $\sphere^2$ in the standard $\R^3$ coordinate system. 
We choose the base point $p$ to be the intersection point of the equator with the positive $y$-axis and let $L_{0}$ be the equator which is a (closed) geodesic through $p$. 
By $L_\theta$ we denote the great circle that is obtained by rotating the equator by a positively oriented rotation around the $y$-axis by an angle $\theta$. We choose the point $q\in \sphere^2$ to be the north pole, $q= N$. Then, there is no unique projection point $P_{0} q$. 
Indeed, each point on the equator is at the same distance to the north 
pole. This means that the only natural extension of $P_{0}$ onto the entire sphere is a multivalued map at the point $q=N$. (Obviously, the same thing is true for the south pole $S$.)
In particular, this means that the measure and dimension of the set $\{N\}$ "explode" under the map $P_0$. On the other hand, there are sets that are dramatically decreased
in dimension under $P_0$: Consider a connected segment $I$ of the great circle $M=\{q\in\sphere^2:\dm(p,q)=\frac{\pi}{2}\}$ that does not contain the north and south poles. 
Then $P_0(I)$ contains only one point, which we will further on denote by $P_0(I)=\{p_0\}$. In particular, $P_0$ has shrunk a set of positive $\Hhh^1$-measure to a single point. 
If we assume in addition, that the segment $I$ is bounded away from the two poles, then there exists a small range of angles $(0,\epsilon)$, $\epsilon>0$, 
such that $P_\theta(I)$ is a one point set for all $\theta\in (0,\epsilon)$. In particular, this shows that Marstrand's theorem does not hold in this setting.
So both the upper and the (generic) lower bound for dimension distortion that hold in $B(p,\frac{\pi}{2})\subset\sphere^2$, fail on $\sphere^2$.

\medskip In fact, the set of angles, for which these exceptional phenoma occur, can be described quite precisely. 
We define the projection $P:[0,\pi)\times \sphere^2\rightarrow L_\theta$ to be the
multivalued map given by $P_\theta(q)=\{ l\in L_\theta: \dm(l,q)\leq\dm(l',q) \text{ for all } l'\in L_\theta\}$.
For  a set $A\subseteq \sphere^2$, we write $P_\theta(A)$ for $\bigcup_{q\in A}P_\theta(q)$. By $\langle\cdot,\cdot\rangle$ we denote the scalar product on $\R^3$. 
Then we can write $M=\{q\in\sphere^2:\langle p,q\rangle=0\}$. Note that on $\sphere^2\backslash M$ the mulitvalued projection $P_\theta$ (applied to points or sets) coincides with the one-valued projections studied in the previous sections. We thus mainly wish to study the projection of subsets of $M$.
\medskip 

For all points $q\in \sphere^2$ and angles $\theta\in[0,\pi)$, it holds that: $P_\theta(q)=L_\theta$ if and only if $\langle q,l \rangle =0$ for each $l\in L_\theta$. Also, 
$P_\theta(q)=\{p_0\}$ if and only if  $\langle q,l \rangle \neq 0$ for some $l\in L_\theta$. 
Note that for all angles $\theta$, there are exactly two points $q\in M$ that satisfy $\langle q,l\rangle =0$ for all $l\in L_\theta$. 
These are $q_\theta:=p \times v_\theta $ and $-q_\theta$, where $\times $ denotes the cross product in $\R^3$. 
Also, for all pairs $\{q,-q\}$ of antipodal points in $M$, there exists exactly one $v_\theta$,
such that $\langle q,v_\theta \rangle =0$. Thus there is a one-to-one correspondence between pairs $\{q,-q\}$ and vectors $v_\theta$ 
with $\theta\in[0,\pi)$. Since the assignment $\theta\mapsto v_\theta$ is unique,
this yields a one-to-one correspondence between pairs $\{q,-q\}$ and  angles $\theta\in[0,\pi)$.
We can consider $M$ to be a copy of $\sphere^1$  isometrically embedded in $\sphere^2$. Thus by identifying each point $q\in M$ with its antipodal point $-q$,
we obtain a new manifold $\tilde{M}$ (we might call it the real projective space of dimension $1$) that itself can be considered to be an isometric copy of $\sphere^1$. Let $u$ denote the projections map $u:M\rightarrow\tilde{M},\ q\mapsto[q]$. 
So the one-to-one correspondence between pairs $\{q,-q\}$ and  angles $\theta\in[0,\pi)$ can be written as a bijection $\psi:\tilde{M}\rightarrow [0,\pi)$, defined by: $\psi([q])=\theta$ if and only if $\langle q,v_\theta \rangle=0$.
The well-definedness of this mapping  follows from the above considerations. So do the following results: \medskip

Let $A\subseteq M$ and by $\tilde{A}$ denote the corresponding set in $\tilde{M}$, i.e. $\tilde{A}=u(A)$. 
 Then \[\{ \theta\in[0,\pi): P_\theta(A)=L_\theta\}=\{ \theta\in[0,\pi): \langle v_\theta,q\rangle=0 \text{ for some }q\in A\}=\psi(\tilde{A})\]
 and \[\{ \theta\in[0,\pi): P_\theta(A)=\{p_0\}\}=\psi(\tilde{M}\backslash \tilde{A}).\]
 Furthermore, 
 \[ \dim \{ \theta\in [0,\pi): P_\theta(A)=L_\theta\} = \dim (\tilde{A})=\dim(A) \] and 
 \[ \dim \{ \theta\in [0,\pi): P_\theta(A)=\{p_0\}\} = \dim (\tilde{M}\backslash \tilde{A}). \]

 Informally speaking, Marstrand's Theorem  says that for a large quantity of angles there is no loss in the dimension of the image of the projection. The above discussion indicates this happens even for set valued projections. It would be very interesting to study these phenomena in a more general context, e.g. for set valued projections in positively curved spaces (with not necessarily constant curvature).\medskip
 
In negatively curved spaces, e.g. in Cartan-Hadamard manifolds, closest point projections are always single valued. It would be of interest to prove results similar to Theorem~\ref{mainthm} in this more general setting. One way to approach this question could be by 
reducing the problem to the constant curvature case via appropriate comparison theorems. 
However, standard comparison theorems from Riemannian geometry, such as the Theorem of  Topogonov or Rauch are not strong enough to imply regularity and transversality properties of projections necessary for the Peres-Schlag theory. 

\bibliography{literature_proj_on_riem}
\bibliographystyle{abbrv}

\end{document}